\documentclass[reqno,10pt,a4paper]{amsart}
\usepackage{srcltx}
\usepackage{graphicx,color,amssymb,latexsym,amsmath,amsfonts}
\usepackage{mathrsfs}

\newtheorem{theorem}{Theorem}[section]

\newtheorem{lemma}{Lemma}[section]

\numberwithin{equation}{section}

\input cyracc.def

\begin{document}

\title[On asymptotically efficient statistical inference]
{On asymptotically efficient statistical inference on a signal parameter}
\author{M. S. Ermakov}

\address
{Institute of Problems of Mechanical Engineering, RAS\\
Bolshoy pr., V.O., 61\\
199178 St. Petersburg, RUSSIA and\\
St. Petersburg State University,
University str. 28,\\
Petrodvoretz, 198504, St. Petersburg, RUSSIA}
\email{erm2512@mail.ru}

\thanks{Research was supported RFFI Grants  11-01-00577 and 11-01-00769}

\maketitle

\section{Introduction}
In the normal approximation zone the lower bounds of asymptotically efficient statistical inference were comprehensively studied.  The local asymptotic minimax Hajek - Le Cam lower bound  \cite{ha, ib, ku, le, str, van} for estimation and Pitman efficiency  \cite{le, bl, str, van} for hypothesis testing are natural measures of efficiency in parametric statistical inference.  In large deviation zone the Bahadur  efficiencies \cite{bah, bi, bl, ch,  ib, van, pu} are the most widespread measures of asymptotic efficiency of tests and estimators. The paper goal is to study the lower bounds of asymptotic efficiency in the zone of moderate deviation probabilities for the problem of statistical inference on a value of parameter of signal observed in Gaussian white noise. Thus, for this problem, we fill in the gap between asymptotic efficiencies given by the normal approximation and the Bahadur asymptotic efficiencies.

For statistical inference on a parameter of distribution of independent sample this problem was considered in \cite{bor, er03, er12, rad, kal}.   The goal of this paper is to obtain similar results for the problem of statistical inference on a signal parameter. The lower bounds of asymptotic efficiency are given for both logarithmic and sharp asymptotics of moderate deviation probabilities of tests and estimators. The problem of asymptotic efficiency in statistical inference on a signal parameter for both the zone of normal approximation and the zone of large deviation probabilities was studied in a large number of papers ( see
 \cite{ib, ku, bu, ga,  ih,  li} and references therein).

The asymptotic equivalence of different statistical models and the model of signal in Gaussian white noise is very popular theme of research \cite{le, str,  br, go}. These results show the tight relation of  these models and the model of signal in Gaussian white noise. From this viewpoint the paper helps to draw the parallel between the results on moderate deviation probabilities of tests and estimators for different models.

The coverage errors of confidence sets have usually small values. The type I error  probabilities are small in hypothesis testing. These problems are prime examples of applications of large and moderate deviation probabilities in statistics. In particular, the lower bounds of asymptotic efficiency of estimators in moderate deviation zone admit the natural interpretation as the lower bounds of asymptotic efficiency in confidence estimation \cite{er03, er12, wo}.

 \smallskip

The lower bounds of asymptotic efficiency in the problem of signal detection are easily deduced from Neyman - Pearson Lemma and are given for completeness. In the case of one-dimensional parameter the proof of lower bounds  in estimation is based on the lower bounds for hypothesis testing. The proof of local asymptotic minimax lower bounds for estimation of multidimensional parameter is obtained by modification of the proof of similar results for independent sample \cite{er12}.

\smallskip
 We use the following notation.  Denote letters $C, c$ positive constants. For any $x\in R^1$ denote $[x]$ the whole part of $x$. For any event $A$ denote $\chi(A)$  the indicator of this event.
The limits of integration are the same throughout the paper. By this reason, we shall omit the limits of integration. We shall write $\int$  instead of $\int\limits_0^1$. For any function $f \in L_2(0,1)$ denote

$$
\|f\|^2= \int f^2(t) \,dt.
$$
Define the function of standard normal distribution
$$
\Phi(x) = \frac{1}{\sqrt{2\pi}}\int\limits_{-\infty}^x \exp\{-t^2/2\}\,dt, \quad x \in R^1.
$$
We shall omit the index of true value of parameter $\theta_0$ and write ${\mathbf E}[\cdot] = {\mathbf E}_{\theta_0}[\cdot]$ and ${\mathbf P}(\cdot) = {\mathbf P}_{\theta_0}(\cdot)$.

\section{  Lower bounds of asymptotic efficiency}

\subsection{Lower bounds of asymptotic efficiency for logarithmic asymptotic of moderate deviation probabilities of tests and estimators}

Let we observe a realization of random process $Y_\epsilon(t), t \in (0,1), \epsilon > 0$, defined by stochastic differential equation
\begin{equation}\label{1.1}
dY_\epsilon(t) = S(t,\theta)\ dt + \epsilon dw(t).
\end{equation}
Here $S \in L_2(0,1)$  is a signal and $dw(t)$   is Gaussian white noise. Parameter $\theta$ is unknown, $\theta \in \Theta$, and $\Theta$
is open set in $R^d$.

Suppose that $S(t,\theta)$ is differentiable on $\theta$ in $L_2(0,1)$ at the point $\theta_0$, that is, there exists function $S_\theta(t,\theta_0)$ such that
\begin{equation}\label{1.2}
\int (S(t,\theta) - S(t,\theta_0) - (\theta - \theta_0)'S_\theta(t,\theta_0))^2 \, dt = o(|\theta - \theta_0|^2).
\end{equation}
Here $(\theta - \theta_0)'S_\theta(t,\theta_0)$
is inner product of  $\theta - \theta_0$ and $S_\theta(t,\theta_0)$.

Fisher information matrix equals
\begin{equation}
I(\theta) = \int S_\theta(t,\theta) S'_\theta(t,\theta)\, dt.
\end{equation}

Make the following assumption.

{\bf A1}. There holds (\ref{1.2}) at the point $\theta_0 \in \Theta$. The Fisher information matrix $I(\theta_0)$  is positively definite.

For the logarithmic asymptotics the problem of lower bounds of efficiency  for large and moderate deviation probabilities of tests and estimators is usually reduced to one-dimensional. Thus, in this section, we suppose $d=1$.

Consider the problem of testing hypothesis $H_0 : \theta = \theta_0$
versus $H_\epsilon : \theta = \theta_\epsilon:=\theta_0 + u_\epsilon$,
where $u_\epsilon>0$, $u_\epsilon \to 0$,
$\epsilon^{-1}u_\epsilon \to \infty$ as~$\epsilon \to 0$.

For any test $K_\epsilon$ denote  $\alpha(K_\epsilon)$ and $\beta(K_\epsilon)$ respectively its type I and type II error probabilities.

Define the test statistics
\begin{equation}
T = I^{-1/2}(\theta_0) \int S_\theta(t,\theta_0)\, dY_\epsilon(t).
\end{equation}

\begin{theorem}\label{t1}  Assume
{\rm A1}. Then, for any family of tests  $K_\epsilon$
such that $\alpha(K_\epsilon)< c < 1 $ and $\beta(K_\epsilon) <c< 1$, there holds
\begin{equation}\label{2.1}
\limsup_{\epsilon \to 0} (\epsilon^{-1} u_\epsilon I^{1/2}(\theta_0))^{-1} (|2\ln \alpha (K_\epsilon)|^{1/2} + |2\ln \beta (K_\epsilon)|^{1/2} ) \le 1.
\end{equation}
The lower bound in {\rm(\ref{2.1})} is attained on a family of tests generated the test statistics $T$.
\end{theorem}

\begin{theorem}\label{t2} Assume
{\rm А1}. Then, for any  estimator $\widehat\theta_{\epsilon}$, there holds
\begin{equation}\label{2.2}
\lim_{\epsilon\to 0} \sup_{\theta=\theta_0, \theta_0 + 2u_\epsilon} \epsilon^2u_\epsilon^{-2}I^{-1}(\theta_0) \ln {\mathbf P}_\theta (|\widehat\theta_\epsilon - \theta| \ge u_\epsilon)  \ge -\frac{1}{2}.
\end{equation}
\end{theorem}

\subsection{Lower bounds of efficiency for sharp asymptotics of moderate deviation probabilities of tests and estimators. The case of one-dimensional parameter }
Fix $\lambda, 0 < \lambda \le 1$.

The results will be proved in the zone $u_\epsilon=o(\epsilon^\frac{2}{2+\lambda})$
if the following assumption holds.

{\bf A2}. There holds
\begin{align}\label{1.3}
\int (S(t,\theta) -S(t,\theta_0) - (\theta- \theta_0)'S_\theta(t,\theta_0))^2\ dt
&= O(|\theta - \theta_0|^{2+\lambda}),
\\
\label{1.4}
\int (S(t,\theta) - S(t,\theta_0))^2 dt - (\theta -
\theta_0)'I(\theta_0)(\theta - \theta_0) &=
O(|\theta-\theta_0|^{2+\lambda}).
\end{align} 

In the case of multidimensional parameter the lower bound in hypothesis testing depends essentially on the geometry of sets of hypotheses and alternatives. We shall consider only the case of one-dimensional parameter. In this case the problem usually is reduced to the problem of testing simple hypothesis versus simple alternative. Consider the problem of testing hypothesis $H_0 : \theta = \theta_0$
versus alternatives $H_\epsilon : \theta =
\theta_\epsilon:=\theta_0 + u_\epsilon$. We suppose additionally that
$\epsilon^{-2}u_\epsilon^{2+\lambda} \to 0$ as~$\epsilon\to0$.

\begin{theorem}\label{t3} Assume
{\rm A1} and {\rm А2}. Let $\epsilon^{-1}u_\epsilon \to \infty$,
$\epsilon^{-2}u_\epsilon^{2+\lambda} \to 0$ as $\epsilon \to 0$.
For any family of tests $K_\epsilon$ such that  $\alpha_\epsilon:=\alpha(K_\epsilon)< c < 1 $, there holds
\begin{equation}\label{2.3}
\beta(K_\epsilon) \ge \Phi(x_{\alpha_\epsilon} - \epsilon^{-1}
u_\epsilon I^{1/2}(\theta_0))(1 + o(1))
\end{equation}
where $x_{\alpha_\epsilon}$ is defined the equation $\alpha_\epsilon = \Phi(x_{\alpha_\epsilon})$.

The lower bound {\rm(\ref{2.3})} is attained on the tests $L_\epsilon$ generated test statistic ~$T$.

If in  {\rm(\ref{2.3})} the equality is attained, then, for the family of tests  $L_\epsilon$,
$\alpha_\epsilon = \alpha(L_\epsilon)$, there hold
\begin{equation}\label{2.3a}
\lim_{\epsilon\to 0}\alpha_\epsilon^{-1}{\mathbf E}_{\theta_0}[|K_\epsilon - L_\epsilon|] =0
\end{equation}
and
\begin{equation}\label{2.3b}
\lim_{\epsilon\to 0}(\Phi(x_{\alpha_\epsilon} - \epsilon^{-1} u_\epsilon I^{1/2}(\theta_0)))^{-1}{\mathbf E}_{\theta_\epsilon}[|K_\epsilon - L_\epsilon|] =0.
\end{equation}
\end{theorem}

\noindent
{\bf Remark.} For $u_\epsilon = \epsilon u, u > 0$,  the lower bound (\ref{2.3}) becomes the lower bound for the Pitman efficiency.

Define the statistic
$$
T_0 = I^{-1/2}(\theta_0) \int S_\theta(t,\theta_0)\, dw(t).
$$

\begin{theorem}\label{t4} Let $d=1$ and let
{\rm A1} and {\rm A2} be satisfied. Let $\epsilon^{-1}u_\epsilon \to \infty, \epsilon^{-2}u_\epsilon^{2+\lambda} \to 0$ as $\epsilon \to 0$. Then, for any estimator $\widehat \theta_\epsilon$, there holds
\begin{equation}\label{2.4}
\lim\inf_{\epsilon\to 0} \sup_{|\theta-\theta_0| < C_\epsilon u_\epsilon} \frac{{\mathbf P}_\theta( |\widehat\theta_\epsilon - \theta| > u_\epsilon)}{2\Phi(-\epsilon^{-1}I^{1/2}(\theta_0) u_\epsilon)} \ge 1
\end{equation}
for any family of constants $C_\epsilon \to \infty$ as $\epsilon \to 0$.

If the equality is attained in {\rm(\ref{2.4})} for $C_\epsilon \to \infty, \epsilon^{-2} C_\epsilon^{2+\lambda}u_\epsilon^{2+\lambda} \to 0$ as $\epsilon \to 0$, then, for any family of parameters $\theta_\epsilon, |\theta_\epsilon - \theta_0| < C_\epsilon u_\epsilon$, there holds
\begin{multline}\label{2.4a}
\lim_{\epsilon \to 0}
\Big(\Phi(-\epsilon^{-1}I^{1/2}(\theta_0) u_\epsilon)\Big)^{-1}
{\mathbf E}_{\theta_\epsilon}
\Big[|\chi(|\widehat\theta_\epsilon - \theta_\epsilon| > u_\epsilon)\\
- \chi(|I^{-1/2}(\theta_0) T_0 - (\theta_\epsilon - \theta_0)| > u_\epsilon)|
\Big] = 0.
\end{multline}
\end{theorem}

Theorems \ref{t1}, \ref{t2}, \ref{t3} and \ref{t4} are versions of Theorems 2.2, 2.5, 2.3 и 2.7, established in  \cite{er03}   for the problems of statistical inference on a parameter of distribution of independent sample.

\subsection{Lower bound of  efficiency for sharp asymptotic of confidence estimation of multidimensional parameter}
For multidimensional parameter we can derive a version of Theorem \ref{t4} with some additional assumptions.

We say that the set $\Omega \subset R^d$ is central-symmetric if $x \in \Omega$ implies $-x \in \Omega$.
Denote $\partial\Omega$ the  boundary of~$\Omega$.

Make the following assumptions.

{\bf A3}
For each $v \in R^d$, there holds
\begin{equation}\label{1.5}
v'I(\theta)v- v'I(\theta_0)v = O(|v|^2 |\theta - \theta_0|^\lambda).
\end{equation}

{\bf A4}. The set $\Omega$ is bounded, convex and central symmetric. The boundary $\partial\Omega$ is $C^2$-manifold. The principal curvatures in each point of $\partial\Omega$ are negative.

\begin{theorem}\label{t5} Assume
{\rm А1--А3} for all $\theta_0 \in \Theta$. Let the set
$\Omega$ satisfies {\rm A4}. Let $\Theta_0$ is open bounded set such that $\partial\Theta_0\! \subset\! \Theta$.
Let $\epsilon^{-1}u_\epsilon \!\to\! \infty$,
$\epsilon^{-2}u_\epsilon^{2+\lambda} \to 0$ as $\epsilon \to 0$. Then, for any estimator $\widehat \theta_\epsilon$, there holds
\begin{equation}\label{2.7}
\liminf_{\epsilon\to 0}\inf_{\theta_0\in\Theta_0}\sup_{|\theta-\theta_0|<C_\epsilon u_\epsilon}\frac{{\mathbf P}_\theta
(I^{1/2}(\theta_0)(\widehat\theta_\epsilon
- \theta) \notin u_\epsilon \Omega)} {{\mathbf P}(\zeta\notin
\epsilon^{-1}u_\epsilon\Omega)}\ge 1
\end{equation}
with $C_\epsilon \to \infty$ as $\epsilon \to 0$. Here $\zeta$ is Gaussian random vector with identity covariance matrix and ${\mathbf E}[\zeta] = 0$.
\end{theorem}

Theorems \ref{t4} and \ref{t5} can be considered as lower bounds of asymptotically efficient confidence estimation. In confidence estimation the covariance matrices of estimators  are often unknown. Then the confidence sets are defined on the base of pivotal statistics. Pivotal statistics are widely applied in hypothesis testing as well. For such a setup one can modify Theorems   \ref{t4} and \ref{t5}. The general approach to such a setup is given in Theorem 2.2 in    \cite{er12}. We do not discuss this problem in details in order not to overload the paper.

\section{Proofs of Theorems \ref{t1}, \ref{t2} and \ref{t3}, \ref{t4}}
Proofs of Theorems \ref{t1} and \ref{t3} are based on straightforward application  of Neyman - Pearson Lemma and analysis of asymptotic distribution of logarithm of likelihood ratio.  Theorems \ref{t2} and \ref{t4} are deduced from Theorems   \ref{t1} and \ref{t3} respectively using the same reasoning  as in \cite{er03}  (Theorems 2.3 and 2.7). In particular,  (\ref{2.4a}) follows from (\ref{2.3a})  and (\ref{2.3b}). By this reason, the proofs of Theorems \ref{t2} and \ref{t4} are omitted. In \cite{er03} similar results were obtained for the problem of statistical inference on a parameter of distribution of independent sample.

The proof will be given for Theorem \ref{t3}. The proof of Theorem \ref{t1} is similar.  It suffices to replace only in all estimates $O(|u_\epsilon|^{2+\lambda})$  with $o(|u_\epsilon|^2)$. Correctness of such a replacement follows from Assumption  A1.

Suppose that the hypothesis is valid. Then the logarithm of likelihood ratio
 (see \cite{ib, ku}) equals
\begin{equation}\label{3.1}
\begin{split}
&L( \theta_0+u_\epsilon,\theta_0)  \\
&:= \epsilon^{-2}\!\!\int (S(t,\theta_\epsilon) \!-\! S(t,\theta_0))\, dY_\epsilon(t)
\!-\! (2\epsilon^2)^{-1}(\|S(t,\theta_\epsilon)\|^2\! -\! \|S(t,\theta_0)\|^2)\\
& =
\epsilon^{-1}\!\!\!\int (S(t,\theta_\epsilon)\! -\! S(t,\theta_0)) dw(t)
\!-\! (2\epsilon^2)^{-1}\|S(t,\theta_\epsilon) \!-\! S(t,\theta_0)\|^2.
\end{split}
\end{equation}
Therefore the test statistics can be defined as
\begin{equation}\label{3.2}
T_{1} =\xi(\theta_\epsilon,\theta_0) =\epsilon^{-1} \int (S(t,\theta_\epsilon) - S(t,\theta_0))\, dY_\epsilon(t).
\end{equation}
 By A2, for the proof of asymptotic efficiency of tests $L_\epsilon$, it suffices to estimate difference of stochastic parts of   $T_1$ and $(\theta_\epsilon - \theta_0)I^{1/2}(\theta_0) T$, defined by statistics
$$
T_{1\epsilon}   =  \epsilon^{-1} \int (S(t,\theta_\epsilon) - S(t,\theta_0))\, dw(t)
$$
and $\epsilon^{-1}I^{1/2}(\theta_0)(\theta_\epsilon -\theta_0) T_0$ respectively.

Denote
\begin{equation*}
\rho^2(\theta_\epsilon,\theta_0) = \|S(t,\theta_\epsilon) - S(t,\theta_0)\|^2.
\end{equation*}
By straightforward calculations, we get
\begin{equation}\label{3.3}
{\mathbf E}_{\theta_0} [T_{1\epsilon}] = 0
\end{equation}
and, by (\ref{1.4}), we have
\begin{equation}\label{3.4}
{\mathbf E}_{\theta_0} [T^2_{1\epsilon}] = \epsilon^{-2}\rho^2(\theta_\epsilon,\theta_0)
=\epsilon^{-2} u_\epsilon^2 I(\theta_0)  + O(\epsilon^{-2} u_\epsilon^{2+\lambda}).
\end{equation}
For the alternative, we get
\begin{align}\label{3.5}
{\mathbf E}_{\theta_\epsilon} [T_{1\epsilon}]& = \epsilon^{-1} {\mathbf E}_{\theta_0} [\xi(\theta_\epsilon,\theta_0) \exp\{\epsilon^{-1}\xi(\theta_\epsilon,\theta_0) - (2\epsilon^2)^{-1}\rho^2(\theta_\epsilon,\theta_0)\}]\notag\\
& =
\epsilon^{-2}\rho^2(\theta_\epsilon,\theta_0)= \epsilon^{-2}(u_\epsilon^2 I(\theta_0) + O(u_\epsilon^{2+\lambda}))
\end{align}
and
\begin{equation}\label{3.6}
\mathbf{Var}_{\theta_\epsilon} [T_{1\epsilon}^2] = \epsilon^{-2}\rho^2(\theta_\epsilon,\theta_0)
=\epsilon^{-2} u^2_\epsilon I(\theta_0) + O(u_\epsilon^{2+\lambda}).
\end{equation}
The lower bound (\ref{2.3}) follows from (\ref{3.1}) - (\ref{3.6})

The proof of asymptotic efficiency of test statistics   $T$ is based on the following lemma.

\begin{lemma} \label{l1} Let
$\vec \eta_\epsilon \!=\! (\eta_{1\epsilon},\eta_{2\epsilon})'$\! be Gaussian random vectors such that
$ {\mathbf E}[\eta_{1\epsilon}] = 0$,
${\mathbf E}[\eta_{2\epsilon}] = 0$,
${\mathbf E}[\xi^2_{1\epsilon}] = 1$,
${\mathbf E}[\xi^2_{2\epsilon}] = O(|u_\epsilon|^\lambda)$,
${\mathbf E}[\eta_{1\epsilon}\eta_{2\epsilon}] = O(|u_\epsilon|^\lambda)$. Then we have
\begin{equation}\label{3.7}
{\mathbf P}(\eta_{1\epsilon} > \epsilon^{-1}u_\epsilon) = {\mathbf P}(\eta_{1\epsilon} + \eta_{2\epsilon} > \epsilon^{-1}u_\epsilon)(1+o(1)).
\end{equation}
\end{lemma}

\begin{proof} Denote $A_\epsilon$ covariance matrix of random vector $\vec \eta_\epsilon$.
Let $\zeta_1$, $\zeta_2$ be independent  random variables having the standard normal distribution. Define the random vector $\vec \zeta = (\zeta_1, \zeta_2)$.
Denote $\vec\omega_\epsilon = (\omega_{1\epsilon},\omega_{2\epsilon})' = A^{1/2}_\epsilon \vec\zeta$. Then we have
\begin{equation}\label{3.8}
{\mathbf P}(\eta_{1\epsilon} + \eta_{2\epsilon} > \epsilon^{-1}u_\epsilon) = {\mathbf P} (\omega_{1\epsilon} + \omega_{2\epsilon} >\epsilon^{-1}u_\epsilon).
\end{equation}
By  straightforward calculation, using the identity $A_\epsilon^{1/2}A_\epsilon^{1/2}=A_\epsilon$, we get that the elements of matrix $A_\epsilon^{1/2} =\{a_{\epsilon,ij}\}_{i,j=1}^2$ have the following orders\break
$a_{\epsilon,22} = O(|u_\epsilon|^{\lambda/2}), a_{\epsilon,12} = O(|u_\epsilon|^{\lambda/2})$.
Hence, using (\ref{3.8}), we get (\ref{3.7}).  

Thus it remains to verify that  the normalized random variables
\begin{equation*}
\eta_{1\epsilon} = u_\epsilon^{-1}\xi(\theta_\epsilon,\theta_0) \quad \mbox{и} \quad
\eta_{2\epsilon} = u_\epsilon^{-1} (\xi(\theta_\epsilon,\theta_0)- u_\epsilon \tau)
\end{equation*}
satisfy the assumptions of Lemma \ref{l1} in the case of hypothesis and alternative. Here
$$
\tau = \tau_{\theta_0} = \int S_\theta(t,\theta_0)\, dw(t).
$$
Suppose the hypothesis is valid.

By (\ref{3.4}), we get
\begin{equation}\label{3.10}
{\mathbf E} \eta_{1\epsilon}^2 = I + O(|u_\epsilon|^\lambda).
\end{equation}
We have
\begin{equation}\label{3.11}
{\mathbf E}[\eta_{1\epsilon}\eta_{2\epsilon}] = u_\epsilon^{-2} \rho^2(\theta_\epsilon,\theta_0)- u_\epsilon^{-1} \int (S(t,\theta_\epsilon) - S(t,\theta_0)) S_\theta(t,\theta_0) \, dt.
\end{equation}
By (\ref{1.3}), we get
\begin{align}\label{3.12}
O(u_\epsilon^{2+\lambda}) &= \|S(t,\theta_\epsilon) - S(t,\theta_0) - u_\epsilon S_\theta(t,\theta_0)\|^2\\
&=
\rho^2(\theta_\epsilon,\theta_0) - 2 u_\epsilon \int (S(t,\theta_\epsilon) - S(t,\theta_0)) S_\theta(t,\theta_0)\, dt + u_\epsilon^2 I(\theta_0).\notag
\end{align}
Hence, by (\ref{1.4}), we get
\begin{equation}\label{3.13}
u_\epsilon \int (S(t,\theta_\epsilon) - S(t,\theta_0)) S_\theta(t,\theta_0)\, dt = u_\epsilon^2 I(\theta_0) + O(|u_\epsilon|^{2+\lambda}).
\end{equation}
By (\ref{1.4}),(\ref{3.11}) and (\ref{3.13}), we get
\begin{equation}\label{3.14}
{\mathbf E}[\eta_{1\epsilon}\eta_{2\epsilon}] = O(|u_\epsilon|^\lambda).
\end{equation}
By (\ref{3.10}) and (\ref{3.14}), we get that the assumptions  of Lemma  \ref{l1}   are satisfied if the hypothesis is valid.

Suppose that the alternative is valid. By straightforward calculations, using (\ref{3.13}), we get
\begin{align}\label{3.15}
{\mathbf E}_{\theta_\epsilon} [\tau] &=
{\mathbf E}_{\theta_0}[\tau \exp\{\epsilon^{-1}\xi(\theta_\epsilon,\theta_0) -(2\epsilon^2)^{-1}
\rho^2(\theta_\epsilon,\theta_0)\}]\notag\\
&=
\epsilon^{-1} \int (S(t,\theta_\epsilon) - S(t,\theta_0)) S_\theta(t,\theta_0)\, dt\\
&=
 \epsilon^{-1}u_\epsilon I(\theta_0) + O(\epsilon^{-1}|u_\epsilon|^{1+\lambda}),\notag
\end{align}
and, arguing similarly, we get
\begin{equation}\label{3.16}
{\mathbf E}_{\theta_\epsilon} [\xi(\theta_0,\theta_\epsilon)] =
\epsilon^{-1}\rho^2(\theta_\epsilon,\theta_0)= \epsilon^{-1}u_\epsilon^2 I(\theta_0) + O(\epsilon^{-1}|u_\epsilon|^{2+\lambda}).
\end{equation}
Using the same technique, we get
\begin{align}\label{3.17}
{\mathbf E}_{\theta_\epsilon} [\xi^2(\theta_\epsilon,\theta_0)] &=
\rho^2(\theta_\epsilon,\theta_0)+\epsilon^{-2}\rho^4(\theta_\epsilon,\theta_0)\\
&=u_\epsilon^2 I(\theta_0) + \epsilon^{-2}u_\epsilon^4 I^2(\theta_0)+ O(|u_\epsilon|^{2+\lambda}+ \epsilon^{-2}|u_\epsilon|^{4+\lambda}),\notag
\end{align}
\begin{equation}\label{3.18}
\begin{split}&
u_\epsilon {\mathbf E}_{\theta_\epsilon} [\xi(\theta_\epsilon,\theta_0) \tau]= u_\epsilon \int (S(t,\theta_\epsilon) - S(t,\theta_0)) S_\theta(t,\theta_0)\, dt (1+ \epsilon^{-2}\rho^2(\theta_\epsilon,\theta_0))\\&=
u_\epsilon^2 I(\theta_0) + \epsilon^{-2}u_\epsilon^4 I^2(\theta_0)+ O(|u_\epsilon|^{2+\lambda}+ \epsilon^{-2}|u_\epsilon|^{4+\lambda})
\end{split}
\end{equation}
and
\begin{equation}\label{3.19}
\begin{split}&
u^2_\epsilon {\mathbf E}_{\theta_\epsilon} [\tau^2] = u_\epsilon^2 I(\theta_0) + \epsilon^{-2}u_\epsilon^2\left(\int (S(t,\theta_\epsilon) - S(t,\theta_0)) S_\theta(t,\theta_0)\, dt\right)^2
\\&=
u_\epsilon^2 I(\theta_0) + \epsilon^{-2}u_\epsilon^4 I^2(\theta_0)+ O(|u_\epsilon|^{2+\lambda}+ \epsilon^{-2}|u_\epsilon|^{4+\lambda}).
\end{split}
\end{equation}
By (\ref{3.17})--(\ref{3.19}), we get
\begin{align*}
{\mathbf E}_{\theta_\epsilon} [\eta_{2\epsilon}^2] &= O(|u_\epsilon|^{\lambda}+ \epsilon^{-2}|u_\epsilon|^{2+\lambda}),
\\
{\mathbf E}_{\theta_\epsilon} [\eta_{1\epsilon}\eta_{2\epsilon}] &= O(|u_\epsilon|^{\lambda}+ \epsilon^{-2}|u_\epsilon|^{2+\lambda}).
\end{align*}
This implies that the assumptions of Lemma \ref{l1} are satisfied in the case of alternative.
\end{proof}

Proof of Theorem \ref{t4} is based on the following version of Theorem \ref{t3}.  In this version the problem of testing hypothesis $H_0: \theta = \theta_0+ C_1 u_\epsilon$ versus alternatives $H_{1\epsilon}: \theta = \theta_0 + C_2 u_\epsilon$ is explored.

\begin{lemma} \label{l2}  Assume
{\rm A1} and {\rm А2}. Then, for any family of tests $K_\epsilon$, such that  $\alpha_\epsilon:=\alpha(K_\epsilon)< c < 1 $, there holds
\begin{equation}\label{2.3db}
\beta(K_\epsilon) \ge \Phi(x_{\alpha_\epsilon} - \epsilon^{-1} (C_2-C_1)u_\epsilon I^{1/2}(\theta_0))(1 + o(1))
\end{equation}
where $x_{\alpha_\epsilon}$ is defined the equation $\alpha_\epsilon = \Phi(x_{\alpha_\epsilon})$.

The lower bound {\rm(\ref{2.3db})} is attained on the tests $L_{\epsilon}$ generated the tests statistics $T$.

If the equality is attained in {\rm(\ref{2.3db})}, then, for the family of tests $L_{\epsilon}$, such that
$\alpha_\epsilon = \alpha(L_{\epsilon})$, there holds
\begin{equation}\label{2.3ab}
\lim_{\epsilon\to 0}\alpha_\epsilon^{-1}{\mathbf E}_{\theta_0}[|K_\epsilon - L_{\epsilon}|] =0,
\end{equation}
\begin{equation}\label{2.3bb}
\lim_{\epsilon\to 0}(\Phi(x_{\alpha_\epsilon} - \epsilon^{-1} (C_2-C_1)u_\epsilon I^{1/2}(\theta_0)))^{-1}{\mathbf E}_{\theta_\epsilon}[|K_\epsilon - L_\epsilon|] =0.
\end{equation}
\end{lemma}

The other reasoning of the proof of Theorem  \ref{t4} identically follows to the proof of Theorem 2.7 in \cite{er03}  and is omitted.

\section{Proof of Theorem \ref{t5}}
In Theorem 2.1 in \cite{er12},  a version of Theorem  \ref{t5} has been proved for confidence estimation of parameter of distribution of independent sample.
The proof of Theorem \ref{t5} is a revised version of the proof of this Theorem.

In what follows, we suppose $\theta_0=0$.

We split the proof on the following steps.

1.Bayes approach. We implement the fact that the Bayes risk does not exceed the minimax one  and reduce the problem to the problem of calculation of asymptotic of Bayes risks. We define uniform Bayes a priori distribution on the lattice $\Lambda_\epsilon$ in the cube $K_{v_\epsilon}\! = (-v_\epsilon,v_\epsilon)^d$,
$v_\epsilon= C_\epsilon u_\epsilon$,
$C_\epsilon \to\infty$, $\epsilon^{-2}(C_\epsilon u_\epsilon)^{2+\lambda} \to 0$.
The lattice spacing equals $\delta_{1\epsilon}=c_{1\epsilon}\epsilon^2u_\epsilon^{-1}$ with $c_{1\epsilon} \to 0,
c_{1\epsilon}^{-3}\epsilon^{-2}u_\epsilon^{2+\lambda} \to 0$ as $\epsilon \to 0$. Denote $l_\epsilon = [v_\epsilon/\delta_{1\epsilon}]$.

\smallskip
2. We split the cube $K_{v_\epsilon}$ into small cubes
$$
\Gamma_{i\epsilon} = x_{\epsilon i}+
(-c_{2\epsilon}\epsilon^2u_\epsilon^{-1},c_{2\epsilon}\epsilon^2u_\epsilon^{-1}]^d,
\quad 1 \le i \le m_n
$$
with $ c_{2\epsilon} \to 0$,
$c_{2\epsilon}c_{1\epsilon}^{-1} \to \infty$,
$c_{1\epsilon}^{-3}\epsilon^{-2}b_\epsilon^{2+\lambda} \to 0$
as $\epsilon \to 0$.

 Using the fact that normalized a posteriori Bayes risk tends to constant in probability as   $\epsilon \to 0$,
 we study the asymptotic of a posteriori Bayes risks independently for each event $W_{i\epsilon}:  
\tau \in \epsilon^{-1}  \Gamma_{\epsilon i}$.

\smallskip
3. We narrow down the set of parameters for a posteriori Bayes risk minimization. We split the lattice $\Lambda_\epsilon$ on the subsets $\Lambda_{ie}, 1 \le i \le m_{2i\epsilon}$. Each set $\Lambda_{ie}$ is a lattice in the union of  finite number of very narrow parallelipipeds   $K_{ije}$. The problem of minimization of a posteriori Bayes risk is solved independently for each set $\Lambda_{ie}$ and the results are added
\begin{equation}\label{3.8a}
\begin{split}&
\inf_{\widehat\theta_\epsilon}\sup_{\theta \in K_{v_\epsilon}} {\mathbf P}_\theta(\widehat\theta_\epsilon
- \theta \notin u_\epsilon\Omega)\\& \ge \inf_{\widehat\theta_\epsilon}
(2l_\epsilon)^{-d}\sum_{i=1}^{m_n}\sum_{\theta \in \Lambda_\epsilon}  {\mathbf P}_\theta(\widehat\theta_\epsilon -\theta \notin u_\epsilon \Omega, W_{i\epsilon})
 \\&\ge
(2l_\epsilon)^{-d}\sum_{i=1}^{m_\epsilon}\sum_{e=1}^{m_{2i\epsilon}}\inf_{\widehat\theta_\epsilon}
\sum_{\theta \in \Lambda_{ie}}  {\mathbf P}_\theta(\widehat\theta_\epsilon - \theta
\notin u_\epsilon \Omega, W_{i\epsilon}).
\end{split}
\end{equation}

4. For estimation of accuracy of linear approximation of stochastic part of logarithm of likelihood ratio we prove the following inequalities (see Lemma\ref{l3} and, for comparison, (3.4) and Lemma 5.3 in \cite{er12}). For any $\theta_j,\theta_k \in \Lambda_{\epsilon} \cap K_{ije}$ and $\kappa > 0$, there holds
\begin{equation}\label{5}
\begin{split}&
 {\mathbf P}(\epsilon^{-1}|\xi(\theta_j,\theta_k) -
(\theta_k-\theta_j)'\tau_{\theta_j}  - \rho'_\epsilon \tau_{\theta_j}| > \kappa, W_{i\epsilon})\\& \le  C
\int\limits_{\Gamma_{\epsilon i} } \exp\left\{-\frac{|t|^2}{2\epsilon^2
\|S_\theta(t,0)\|^2}\right\} dt \exp\left\{-c\kappa^2|\theta_k -
\theta_j|^{-2-\lambda}\epsilon^{-2}\right\}
\end{split}
\end{equation}
with
\begin{multline*}
\rho_\epsilon= \rho_\epsilon(\theta_j,\theta_k) = \epsilon^2 \|S_\theta(t,\theta_j)\|^{-2}\int S_\theta(t,\theta_j) (S(t,\theta_k)
\\
 - S(t,\theta_j)- (\theta_k-\theta_j)' S_\theta(t,\theta_j) ) \ dt.
\end{multline*}
Since  $\tau \in \epsilon^{-1} \Gamma_{i\epsilon}$, it is easy to show that \begin{equation}\label{6}
\rho'_\epsilon \int S_\theta(t,0) dw(t) < \delta_\epsilon \to 0
\end{equation}
as $\epsilon \to 0$.

\smallskip
5. The estimates of (\ref{5}) and (\ref{6}) -- type  and  "chaining" method allow to implement to addendums of right-hand side of (\ref{3.8a}) the technique of the proof of multidimensional local asymptotic minimax lower bound \cite{ib} on the base of the same reasoning as in \cite{er12}.

For clarity, the definition of parallelepipeds $K_{ij}$ will be provided with $x_{i\epsilon}$  parallel the first ort $e_1$ of coordinate system. Define the subspace $\Pi_1$ orthogonal $e_1$.
Define in the lattice $\Lambda_\epsilon\cap \Pi_1$  sublattice
$\Lambda^1_{i}=\{\theta_{ij}\}_{1\le j \le m_{1i\epsilon}}$ with spacing $2c_{3\epsilon}\delta_{1\epsilon}$ where $c_{3\epsilon}$ is such that
$c_{3\epsilon}/c_{2\epsilon} \to \infty$,
$c_{3\epsilon}\delta_{1\epsilon} = o(\epsilon^2 u_\epsilon^{-1})$ and $c_{3\epsilon}^3c_{1\epsilon}^{-3}\epsilon^{-2}u_\epsilon^{2+\lambda} \to 0$
as $\epsilon \to 0$.

Denote
\begin{align*}
& K_{ij}
=K(\theta_{ij})=\Big\{x: x = \lambda x_{i\epsilon} + u +\theta_{ij},\
 u =\{u_k\}_{k=1}^d,\  u \bot x_{ni},
 \\
 & |u_k| \le c_{3n} \delta_{1n}, \
\lambda \in R^1, \ u \in R^d\Big\} \cap K_{v_\epsilon},\quad 1 \le j \le m_{1i\epsilon}.
\end{align*}
We define  the  sets $\Lambda_{ie}$ for the most simple geometry such that the distance of the set  $\partial\Omega$ to zero is attained only in two points.
Each set $\Lambda_{ie}$ consists of subsets $K(\theta_{ij})\cap \Lambda_{\epsilon}$ such that
\begin{equation*}
\begin{split}
\theta_{ij} \in \Theta_{ie}=\Theta_{i}(k_1,\ldots,k_{d-d_1}) & = \{\theta: \theta=
\theta_{ij}  + (-1)^{t_2} 2k_2 c_{3\epsilon}\delta_{1\epsilon}e_{2}\\& +
\cdots + (-1)^{t_{d}} 2k_{d}c_{3\epsilon}\delta_{1\epsilon} e_d;\,\,
t_2,\ldots t_{d}= 0, 1\}
\end{split}
\end{equation*}
 where $k_2,\ldots,k_{d}$ are fixed for each $\Lambda_{ie}$ and $0 \le k_2,\ldots,k_{d} < C_{1\epsilon}$ with  $C_{1\epsilon}c_{3\epsilon}c_{1\epsilon}  \to \infty, \epsilon^{-2}C_{1\epsilon}^{3}c_{3\epsilon}^{3}c_{1\epsilon}^{3}u_\epsilon^{2+\lambda} \to 0$ as $\epsilon \to 0$.

 Denote $K_{ie} = \bigcup\limits_{\theta\in\Theta_{ie}} K(\theta)$.

For arbitrary geometry of $\partial\Omega$ the definition of sets $\Lambda_{ie}$ is more complicated and, moreover, the indexation becomes cumbersome  (see \cite{er12}). However the reasoning remains almost unchanged.

Fix $\delta > 0$. For all $\theta\in \Lambda_{ie}$ define the events $A_i(0,\theta,\delta) : \epsilon^{-1}(\xi(0,\theta) - \theta'\tau) > \delta$.
Denote $A_{ie} =\bigcap\limits_{\theta\in\Theta_{ie}} A_{i}(0,\theta,\delta)$.
 Denote $B_{ie}$ the additional event for $A_{ie}$.

Arguing similarly to  (3.8) and (3.9) in \cite{er12}, we get
 \begin{equation}
\label{3.21}
\begin{split}&
\inf_{\widehat\theta_\epsilon} \sum_{\theta \in \Lambda_{ie}}
{\mathbf P}_\theta(\widehat\theta_\epsilon - \theta \notin u_\epsilon \Omega,
W_{i\epsilon})
\\
&\ge
\inf_{\widehat\theta_\epsilon} \sum_{\theta \in \Lambda_{ie}}
{\mathbf E} \Big[\chi(\widehat\theta_\epsilon - \theta \notin u_\epsilon
\Omega) \exp\{\epsilon^{-1}\tau - (2\epsilon^2)^{-1}\rho^2(\theta_0,\theta_\epsilon)\},
\\
&\qquad\quad W_{i\epsilon},  A_i(\theta,0,\kappa)\Big]
\\
& \ge {\mathbf E} \bigg[\inf_t \sum_{\theta \in
\Lambda_{ie}} \chi(t - \theta \notin u_\epsilon \Omega)
\exp\left\{\epsilon^{-1}\theta'\tau  -(2\epsilon^{2})^{-1} \theta'I \theta  + o(1)\right\},
\\
&\qquad\quad W_{i\epsilon}, A_{ie} \Big]
 = R_\epsilon
\end{split}
\end{equation}
if $\delta=\delta_\epsilon \to 0$ sufficiently slowly as $\epsilon \to 0$.

Denote $\Delta_{\epsilon} = \exp\{\tau'\tau/2\}$,
$y= y_\theta =\epsilon^{-1}\theta - \tau$. Using\\
 $\epsilon^{-2}u_\epsilon\delta_{1\epsilon} \to 0$,
 $\epsilon^{-2}u_\epsilon^{2 +\lambda} \to 0$ as $\epsilon \to 0$, we  get
\begin{equation}
\label{3.22}
\begin{split}
(2l_\epsilon)^{-d}R_\epsilon & \ge (2l_\epsilon)^{-d}{\mathbf E} \Biggl[\Delta_{\epsilon} \inf_t \sum_{\theta \in
\Lambda_{ie}} \chi(t \!-\! y_\theta \!-\! \tau\!\notin\!
\epsilon^{-1}u_\epsilon \Omega)\exp\biggl\{-\frac{1}{2} y_\theta'I
y_\theta\biggr\},\\
&\qquad W_{i\epsilon}, A_{ie}\Biggr] (1 +o(1)) \\&
= (2v_\epsilon)^{-d}{\mathbf E} \biggl[\!\Delta_{\epsilon} \inf_t \!\int\limits_{\epsilon^{-1}K_{ie} -
\psi_\epsilon} \!\chi(t \!-\! y \!\notin\! \epsilon^{-1}u_\epsilon
\Omega)\exp\biggl\{-\frac{1}{2} y'I y\biggr\}\,dy,\\
&\qquad W_{i\epsilon}, A_{ie}\biggr] (1 +o(1)):=
(2v_\epsilon)^{-d}I_{ie\epsilon}(1 +o(1)).
\end{split}
\end{equation}
For $\kappa \in (0,1)$,  denote
\begin{align*}
K_{i\kappa}&(\theta_{ij}) = \Big\{x: x = \lambda x_{i\epsilon} +u +
\theta_{ij}, \ u=\{u_k\}_1^d,
\\
& |u_k| \le (c_{3\epsilon} - Cc_{2\epsilon})\delta_{1\epsilon},\
 u \bot x_{i\epsilon},  \lambda \in R^1\Big\} \cap K_{(1-\kappa)v_\epsilon},
\end{align*}
\begin{align*}
 K_{ie\kappa} = \bigcup_{\theta \in \Theta_{ie}} K_{i\kappa}(\theta).
\end{align*}
Here $ \bot x_{i\epsilon}$ denotes that the vectors
$u$ and $x_{i\epsilon}$ are orthogonal.

If $\tau \in \epsilon^{-1}\Gamma_{i\epsilon}$, then $\epsilon^{-1} K_{ie\kappa}
\subset \epsilon^{-1}K_{ie} - \tau$
and therefore
\begin{equation}
\label{3.23}
I_{ie\epsilon} \ge U_{ie\epsilon} \overline J_{ie\epsilon} (1+o(1))
\end{equation}
with
$$
U_{ie\epsilon}= {\mathbf E}\left[\Delta_{\epsilon}, W_{i\epsilon}, A_{ie}
\right],
$$
$$
\overline J_{ie\epsilon}:= \inf_t J_{ie\epsilon}(t):=  \inf_t
\int\limits_{\epsilon^{-1}K_{ie\kappa}} \chi(t - y \notin \epsilon^{-1}u_\epsilon \Omega )
\exp\left\{-\frac{1}{2} y'I y \right\}dy.
$$
By Lemma 3.1 in \cite{er12}, we have
\begin{equation}
\label{3.24}
\overline J_{ie\epsilon} = J_{ie\epsilon}(0).
\end{equation}
We have
\begin{equation}
\label{3.25}
{\mathbf E}\left[\Delta_{\epsilon}, W_{i\epsilon}
\right]= \mbox{mes}(\Gamma_{i\epsilon})(1 + o(1)).
\end{equation}
Thus, for the proof of Theorem \ref{t5}, it remains to prove only
\begin{multline}
\label{3.26}
U_{2ie\epsilon}:=  {\mathbf E}\left[\Delta_{\epsilon}, W_{i\epsilon}, B_{ie}
\right] =\exp\{ \epsilon^{-2}|x_{i\epsilon}|^2/2\} {\mathbf P}(W_{i\epsilon}, B_{ie})
\\
= o(\mbox{mes}(\Gamma_{i\epsilon})).
\end{multline}
Then, by (\ref{3.8a}) and (\ref{3.21})-(\ref{3.26}), we get Theorem \ref{t5}.

Thus it remains to estimate ${\mathbf P}(W_{i\epsilon}, B_{ie})$.  For estimation we implement the "chaining"  method.

To simplify the notation we suppose $l_\epsilon = 2^m$.
Fix $\theta \in \Theta_{ie}$.
Define the sets $\Psi_j$, $j=0,1,2,\ldots,m,$ by induction.
We put $\Psi_0 = \{\theta\}$. Define the set $$
\Psi_j= \{\theta: \theta= \theta_{j-1} \pm v_\epsilon 2^{-j} x_{i\epsilon}, \theta_{j-1} \in \Psi_{j-1}\},\quad 1 \le j \le m.
$$
We put $\Psi_{m+1} = \Lambda_{ie\epsilon} \setminus
\bigcup\limits_{j=1}^m \Psi_j$. For each $\theta_j \in \Psi_j$, denote  $\theta_{j-1}$ the nearest point of $\theta  \in \Psi_{j-1}$ that lies between zero and $\theta_j$.

We have
\begin{align}
\label{3.27}
S(\theta_j,0)&=\xi(\theta_j,0) - (\theta_j -\theta)'\tau\notag\\
&= S_1(\theta_j,\theta_{j-1}) + S(\theta_{j-1},0)  + S_2(\theta_j,\theta_{j-1})
\end{align}
with
\begin{equation}
S_1(\theta_j,\theta_{j-1})= \xi(\theta_j,\theta_{j-1}) - (\theta_j -\theta_{j-1})'\tau_{\theta_{j-1}}
\end{equation}
and
\begin{equation}
S_2(\theta_j,\theta_{j-1})= (\theta_j -\theta_{j-1})'(\tau_{\theta_{j-1}} - \tau).
\end{equation}
Then
\begin{equation}\label{3.48}
{\mathbf P}(W_{i\epsilon}, B_{ie}) \le C  \left(\sum_{\theta_0\in\Theta_{ie}}\left( V_{0}(\theta_0) +
\sum_{\theta\in \Lambda_{1ile}(\theta_0) }(V_{1}(\theta)
+V_{2}(\theta))\right)\right)
\end{equation}
with $\Lambda_{1ie}(\theta_0) = \Lambda_{ie}(\theta_0) \setminus \Theta_{ie}$,
\begin{align*}
V_0(\theta_0) &= {\mathbf P}(|S(0,\theta_0) >\delta/4, W_{i\epsilon}),
\\
V_s(\theta_j) &= {\mathbf P}(j^2|S_s(\theta_j,\theta_{j-1})| >\delta/4, W_{i\epsilon}), \quad s=1,2.
\end{align*}
\begin{lemma}\label{l2} We have
\begin{equation}\label{3.49}
V_0(\theta) < C\exp\{-C|\theta|^{-2-\lambda}\delta^2\epsilon^{-2}\} {\mathbf P}(W_{i\epsilon}).
\end{equation}
For $\theta_j \in \Psi_j$, we have
\begin{equation}\label{3.50}
V_s(\theta_j)< C\exp\{-C|\theta_j-\theta_{j-1}|^{-2}u_\epsilon^{\lambda}\delta^2 j^{-4}\epsilon^{-2}\} {\mathbf P}(W_{i\epsilon})
\end{equation}
for $s=1,2$.
\end{lemma}
Substituting  (\ref{3.49}) and (\ref{3.50}) in  (\ref{3.48}), we get (\ref{3.26}).

Proofs of  (\ref{3.49}) and (\ref{3.50}) are akin to the proofs of
(5.6) and (5.7)  in \cite{er12} and are based on the same estimates of Lemmas  5.4 -- 5.8 in \cite{er12}. The proofs  of these Lemmas for the setup of this paper do not differ from the proof in \cite{er12}. We omit the proofs of versions of Lemmas  5.4 - 5.6 in \cite{er12} for this setup. The versions of Lemmas 5.7 and 5.8 in \cite{er12} and their proofs will be only given.

Denote $\overline h = \theta_j - \theta_{j-1}, h= \theta_j, h_1 = \theta_{j-1}$.
 \begin{lemma}\label{l3} For any $u \in R^d$, we get
\begin{equation}\label{3.51}
{\mathbf E}[(u'(\tau - \tau_h))^2] = O(|u|^2 |h|^\lambda).
\end{equation}
\end{lemma}
\begin{lemma}\label{l4} Let $v \perp \overline h, v \in R^d$. Then we have
\begin{equation}\label{3.52}
{\mathbf E}[(\overline h\,'(\tau_{h_1} - \tau))(v'\tau)] = O(|v| |\overline h| |h_1|^{\lambda/2}).
\end{equation}
If $v \parallel \overline h$, then
\begin{equation}\label{3.53}
{\mathbf E}[(\overline h'(\tau_{h_1} - \tau))(v'\tau)] = O(|v| |\overline h| |h_1|^{\lambda}).
\end{equation}
\end{lemma}

\begin{proof}[Proof of Lemma \ref{l3}] Using (\ref{1.3}), we get
\begin{equation}\label{3.54}
\begin{split}
J(h,u) &:= {\mathbf E}[(\xi(h,h+u) - \xi(0,u))^2]
\\
&  \le C({\mathbf E}[(\xi(\theta_0,h+u) - (h + u)'\tau)^2]\\& +
 {\mathbf E}[(\xi(\theta_0,h+u) - u'\tau)^2] + {\mathbf E}[(\xi(\theta_0,h) - h'\tau)^2]) \\& \le
 C(|h+u|^{2+\lambda} + |h|^{2+\lambda} + |u|^{2+\lambda}).
\end{split}
\end{equation}
At the same time, there holds
\begin{equation}\label{3.55}
\begin{split}
{\mathbf E}[(u'(\tau - \tau_h))^2] &\le C(E[(\xi(h,h+u)- u'\tau_h - \xi(0,u) + u'\tau)^2]
\\
& + J(h,u))) 
 \le
C ({\mathbf E}[(\xi(h,h+u)- u'\tau_h)^2]
\\
& + {\mathbf E}[(\xi(0,u))^2 - u'\tau)^2] + J(h,u))\\
& \le
C(|h+u|^{2+\lambda} + |h|^{2+\lambda} + |u|^{2+\lambda}).
\end{split}
\end{equation}
Putting $|u| = C|h|$,  we get (\ref{3.51}).
\end{proof}

\begin{proof}[Proof of Lemma \ref{l4}]
Implementing  Cauchy inequality and Lemma  \ref{l3}, we get
\begin{align}\label{3.56}
|{\mathbf E}[(\overline h'(\tau_{h_1} - \tau))(v'\tau)]|
& \le ({\mathbf E}[((\overline h'(\tau_{h_1} - \tau)))^2])^{1/2} ({\mathbf E}[(v'\tau)^2])^{1/2}
\notag\\
& = O(|v| |\overline h| |h_1|^{\lambda/2}).
\end{align}
Let us prove (\ref{3.53}). We have
\begin{equation}\label{3.57}
O(|v|^2|h|^\lambda) = {\mathbf E}[(v'(\tau - \tau_h))^2]
= v'I(0)v + v' I(h)v - 2{\mathbf E}[(v'\tau)(v'\tau_h)].
\end{equation}
Hence, using(\ref{1.4}), we get
\begin{equation}\label{3.58}
{\mathbf E}[(v'\tau)(v'\tau_h)] = v'I(0)v + O(|v|^2|h|^\lambda).
\end{equation}
Hence
\begin{equation}
\begin{aligned}
{\mathbf E}[(\overline h'(\tau_{h_1} - \tau))(v'\tau)] &= C({\mathbf E}[(h'\tau_{h_1})(h\tau)]
 -  {\mathbf E}[(h'\tau)(h\tau)])
\\
& = O(|v| |\overline h| |h_1|^{\lambda}).
\end{aligned} \label{3.59}
\end{equation}
\end{proof}

\bigskip
 Ermakov M. S. On asymptotically efficient statistical inference on a signal parameter.

\smallskip
We consider the problems of confidence estimation and hypothesis testing on a parameter of signal observed in Gaussian white noise. For these problems we point out lower bounds of asymptotic efficiency in the zone of moderate deviation probabilities. These lower bounds are versions of local asymptotic minimax Hajek--Le Cam lower bound in estimation and the lower bound for Pitman efficiency in hypothesis testing. The lower bounds were obtained for both logarithmic and sharp asymptotics of moderate deviation probabilities.


\begin{thebibliography}{99}

\bibitem{ha} J. Hajek,
{\it Local asymptotic minimax and admisibility in estimation}.
--- In: Proc. Sixth Berkeley Symp. Math. Statist. Probab.  California Univ.
Press, Berkeley \textbf{1} (1972), 175--194.

\bibitem{ib} I. A. Ibragimov, R. Z. Has'minskii,
{\it Statistical estimation: asymptotic theory}. Springer, NY, 1981.

\bibitem{ku}  Yu. A. Kutoyants,
{\it Identification of Dynamical System with Small Noise}.   Springer, Berlin, 1994.

\bibitem{le} L. Le Cam, {\it Limits of experiments.}
--- In: Proc. Sixth Berkeley Symp. Math. Statist. Probab.   California Univ.
Press.  Berkeley {\bf 1} (1972), 245--261.

\bibitem{str} H. Strasser,  {\it Mathematical Theory of Statistics}.
 W. de Gruyter, Berlin, 1985.

 \bibitem{van} A. W. van der Vaart,   {\it Asymptotic Statistics}.
Cambridge. University Press, 1998.

 \bibitem{bl}     R. E. Blahut,
 {\it Hypothesis testing and information theory}.
 --- IEEE Trans. Inform. Theory {\bf 20} (1974), 405--415.

\bibitem{bah}  R. R. Bahadur,
{\it Asymptotic efficiency of tests and estimates}.
---   Sankhy\= a  {\bf 22} (1960), 229--252.

\bibitem{bi} J. Bishwal,
{\it Parameter Estimation of Stochastic Differential Equations}.
 Springer, Berlin-Heidelderg,  2008.

\bibitem{ch} T. Chiyonobu,
{\it Hypothesis testing for signal detection problem and large deviations}.
--- Nagoya Math. J. {\bf 162} (2003), 187--203.

 \bibitem{pu} A. Puhalskii, V. Spokoiny,
 {\it On large-deviation efficiency in statistical inference}.
 --- Bernoulli  {\bf 4} (1998), 203--272.

 \bibitem{bor}  A. A. Borovkov, A. A. Mogulskii,
 {\it Large deviations and hypothesis testing}.
 --- Proceedings of Institute of Mathematics SD RASc,
Nauka, Novosibirsk, 1992.

\bibitem{er03}  M. S.Ermakov,
{\it Asymptotically efficient statistical inference for moderate deviation probabilities.}
--- Theory Probabil. Appl.  {\bf 48} (2003), 676--700.

\bibitem{er12} M. S. Ermakov,
{\it The sharp lower bound  of asymptotic efficiency of estimators
in the zone  of moderate deviation probabilities}.
--- Electronic J. Statist.  {\bf 6}, (2012), 2150--2184.

\bibitem{rad} M. Radavi\v cius,
{\it From asymptotic efficiency in minimax sense to Bahadur efficiency}.
 --- In: New Trends Probab. Statist. (V. Sazonov, T. Shervashidze Eds.),
 VSP/Mokslas, Vilnius  (1991), 629--635.

\bibitem{kal}  W. C. M. Kallenberg,
{\it Intermediate efficiency, theory and examples}.
--- Ann. Statist.  {\bf 11} (1983), 170--182.

\bibitem{bu} M. V. Burnashev,
{\it On maximum likelihood estimator of signal parameter in Gaussian white noise }.
--- Probl. Inf. Transm.   {\bf 11} (1975), 55--69.

\bibitem{ga} F. Q. Gao, F. J. Zhao,
{\it Moderate deviation and hypothesis testing for signal detection problem}.
--- Sci. China. Math.  {\bf 55} (2012), 2273--2284.

\bibitem{ih} S. Ihara, Y. Sakuma,
{\it Signal detection in white Gaussian channel}.
--- In: Proc. Seventh Japan--Russian Symp.   Probab. Theory,
Math. Stat.,  World Scientific, Singapore (1996), 147--156.

\bibitem{li} R. S. Liptzer, A. N. Shiriaev,
{\it Statistics of Random Processes}. Springer, Berlin-Heidelderg, 2005.

 \bibitem{br}  L. D. Brown, A. V. Carter, M. G. Low, C. H. Zhang,
 {\it Equivalence theory for density  estimation,
 Poisson processes and Gaussian white noise with drift}.
 --- Ann. Statist.  {\bf 32} (2004), 2074--2097.

 \bibitem{go}   G. K. Golubev, M. Nussbaum, H. H. Zhou,
 {\it Asymptotic equivalence of spectral density  estimation and Gaussian white noise}.
 --- Ann. Statist. {\bf 38} (2010), 181--214.

\bibitem{wo}   J. Wolfowitz,
{\it Asymptotic efficiency of the maximum likelihood estimator}.
--- Theory Probab. Appl. {\bf 10} (1965), 267--281.

 \end{thebibliography}
\end{document}